\documentclass[sn-mathphys,Numbered]{sn-jnl}

\def\nerve{\mathrm{N}}
\def\cspace{\mathrm{B}}
\def\cechnerve{\mathcal{N}}
\newcommand{\ob}{\operatorname{Ob}}
\newcommand{\mor}{\operatorname{Mor}}
\newcommand{\diag}{\operatorname{diag}}
\newcommand{\op}{^{\operatorname{op}}}

\def\from{:}
\newcommand{\catname}[1]{\mathbf{#1}}
\newcommand{\Set}{\catname{Set}}
\newcommand{\sSet}{\catname{sSet}}
\newcommand{\Cat}{\catname{Cat}}

\usepackage{tikz-cd} 
\usetikzlibrary{cd}

\usepackage{graphicx}%
\usepackage{float}
\usepackage{multirow}%
\usepackage{amsmath,amssymb,amsfonts}%
\usepackage{amsthm}%
\usepackage{mathrsfs}%
\usepackage[title]{appendix}%
\usepackage{xcolor}%
\usepackage{textcomp}%
\usepackage{manyfoot}%
\usepackage{booktabs}%
\usepackage{algorithm}%
\usepackage{algorithmicx}%
\usepackage{algpseudocode}%
\usepackage{listings}%

\theoremstyle{thmstyleone}%
\newtheorem{theorem}{Theorem}
\newtheorem{proposition}[theorem]{Proposition}%
\newtheorem{lemma}[theorem]{Lemma}
\newtheorem{corollary}[theorem]{Corollary}

\theoremstyle{thmstyletwo}%
\newtheorem{remark}{Remark}%

\theoremstyle{thmstylethree}%

\begin{document}

\title[Article Title]{The bifiltration of a relation and extended Dowker duality}


\author*[1]{\fnm{Melvin} \sur{Vaupel}}\email{melvin.vaupel@ntnu.no}

\author[1]{\fnm{Benjamin} \sur{Dunn}}\email{benjamin.dunn@ntnu.no}

\affil[1]{\orgdiv{Department of Mathematical Sciences}, \orgname{Norwegian University of Science and Technology}, \orgaddress{\street{Alfred Getz' vei 1}, \postcode{7036}, \city{Trondheim}, \country{Norway}}}


\abstract{We explain how homotopical information of two composeable relations can be organized in two simplicial categories that augment the relations row and column complexes. We show that both of these categories realize to weakly equivalent spaces, thereby extending Dowker's duality theorem. We also prove a functorial version of this result. Specializing the above construction a bifiltration of Dowker complexes that coherently incorporates the total weights of a relation's row and column complex into one single object is introduced. This construction is motivated by challenges in data analysis that necessitate the simultaneous study of a data matrix rows and columns. To illustrate the applicability of our constructions for solving those challenges we give an appropriate reconstruction result.}

\keywords{Dowker complexes, Simplicial categories, Applied topology, Data analysis}



\maketitle

\section{Introduction}
Data often comes in the form of a matrix. For example can the elements of a pointcloud~$X \subseteq \mathbb{R}^n$ be arranged as column vectors. In studying such data sets one is early on faces a decision: \textit{should I study the matrix's colums or it's rows}? Scientists in different areas have learned through experience which combinations of data and question lend themselves to which modes of analysis. On the other hand, studying a data matrix's rows and columns may often provide complementary information and it might all too often make most sense to use all of it. In doing so one is however confronted with a problem of \textit{coherence}. How does information obtained from the rows of a data matrix fit together with information obtained from it's columns? We wish to study this problem through the lens of filtrations on \textit{Dowker complexes} of a relation~$A : I \times J \rightarrow \{0,1\}$, i.e a binary matrix. \\ 
In \cite{dowker1952homology} Dowker first explained how to associate two simplicial complexes to such a relation.
\begin{enumerate}
    \item The \textit{row complex}~$R(A)$ with simplices the collections~$\sigma \subseteq I$ such that there is some~$j \in J$ where $A(i,j)=1$ for all~$i \in \sigma$.
    \item The \textit{column complex}~$C(A)$ with simplices the collections~$\tau \subseteq J$ such that there is some~$i \in I$ where $A(i,j)=1$ for all~$j \in \tau$.
\end{enumerate}
He then proved that these two complexes always have the same simplicial homology and cohomology. His result, today known as \textit{Dowker duality} was later strengthened to a homotopy equivalence between the row and column complexes realizations by Björner \cite{bjorner1995topological}. In recent years interest in Dowker complexes surged when their applicabilty to problems in applied topology became apparent. Chowdhury and Mémoli gave a functorial version of Dowker duality \cite{chowdhury2018functorial} and Virk generalised their result, while establishing a connection to the functorial nerve lemma and Vietoris Rips filtrations \cite{Virk}. In \cite{robinson2022cosheaf} Robinson explained how to augment Dowker complexes into cosheaves, that with their costalks track witnesses to respective simplices. 
The present paper relates to both Virks and Robinsons lines of work and is based on two motivations. The first is more theoretical, while the second is more applied. Both however start with the nerve lemma and it's connection with Dowker complexes. \\
The nerve lemma tells us that we can reconstruct the homotopy type of a space~$X$ from the \v{C}ech nerve~$\cechnerve \mathcal{U}$ of a good open cover~$\mathcal{U}=\{U_i \subseteq X\}_{i\in I}$. In a way this result is not surprising. The condition that~$\mathcal{U}$ is good means that all intersections of coverelements are either contractible or empty i.e that the homtopical content of~$X$ is concentrated in the combinatorics of how cover elements overlap. These combinatorics are precisely what is encoded in the nerve~$\cechnerve \mathcal{U}$, which has as simplices the collections of cover elements with non-empty overlap. If we drop the assumption of~$\mathcal{U}$ being good, then the nerve lemma does not apply anymore. In order to recover the homotopy type of~$X$ from such a cover in general we have to keep track of the, now possibly non-contractible, homotopical content of the overlaps. One way to do this is by forming a topological category
\[
    X_{\mathcal{U}} = \left[
\begin{tikzcd}
 \bigsqcup_{\sigma \subseteq I} U_{\sigma} & \bigsqcup_{\sigma' \subseteq \sigma} U_{\sigma} \arrow[l,shift left,"\operatorname*{id}"] \arrow[l,shift right,"\iota"']
\end{tikzcd}
\right],
\]
where we denote by $U_{\sigma}$ the intersection~$\bigcap_{i \in \sigma} U_{i}$. The source and target maps in this category are given by identity and inclusion in the sense that a morphism~$x \in U_{\sigma}$ in the component indexed by~$\sigma' \subseteq \sigma$ is mapped to~$x \in U_{\sigma}$ by the source map~$\iota$ and to~$x \in U_{\sigma'}$ by the target map~$\operatorname*{id}$. 
In \cite{Segal} Segal showed that the classifying space of~$X_{\mathcal{U}}$ recovers the homotopy type of the base space if~$\mathcal{U}$ admits a subordinate partition of unity. Later Dugger and Isaksen generalized this result to all open covers \cite{Dugger}. \\
We may link this theory of open covers to that of relations and Dowker complexes by defining for an open cover~$\mathcal{U}=\{U_i \subseteq X\}_{i \in I}$ the relation
\begin{align*}
    & A_{\mathcal{U}} \from I \times X \to \{0,1\} \\
    & A_{\mathcal{U}} = \begin{cases}1 \text{ if } x \in U_i \\
    0 \text{ else} \end{cases}.
\end{align*}
The row complex of this relation is precisely the \v{C}ech nerve of the cover~$\cechnerve{\mathcal{U}}$, with simplices the non-empty intersections of coverelements~$\{\sigma \subseteq I \mid \bigcap_{i \in \sigma} U_{\sigma} \neq \emptyset\}$, while the column complex is the covers \textit{Vietoris complex}. This setting was one of Dowkers original motivations and allowed him to specialize his duality result to an isomorphism between the \v{C}ech and Vietoris (co-)homology groups of a cover. \\
We will now turn towards our two motivations, starting with the more theoretical one. 
\paragraph*{Motivation 1} Dowkers original proof of the nerve lemma relies on carefully constructing appropriate subdivisions and then using that contiguous simplicial maps are homotopic upon realisation. Björner proof from \cite{bjorner1995topological}, as illustrated in Figure \ref{Fig:DowkerProof}, is much shorter and goes as follows. 
\begin{figure}[H]
    \centering    \includegraphics[width=0.6\textwidth,height=0.6\textheight,keepaspectratio]{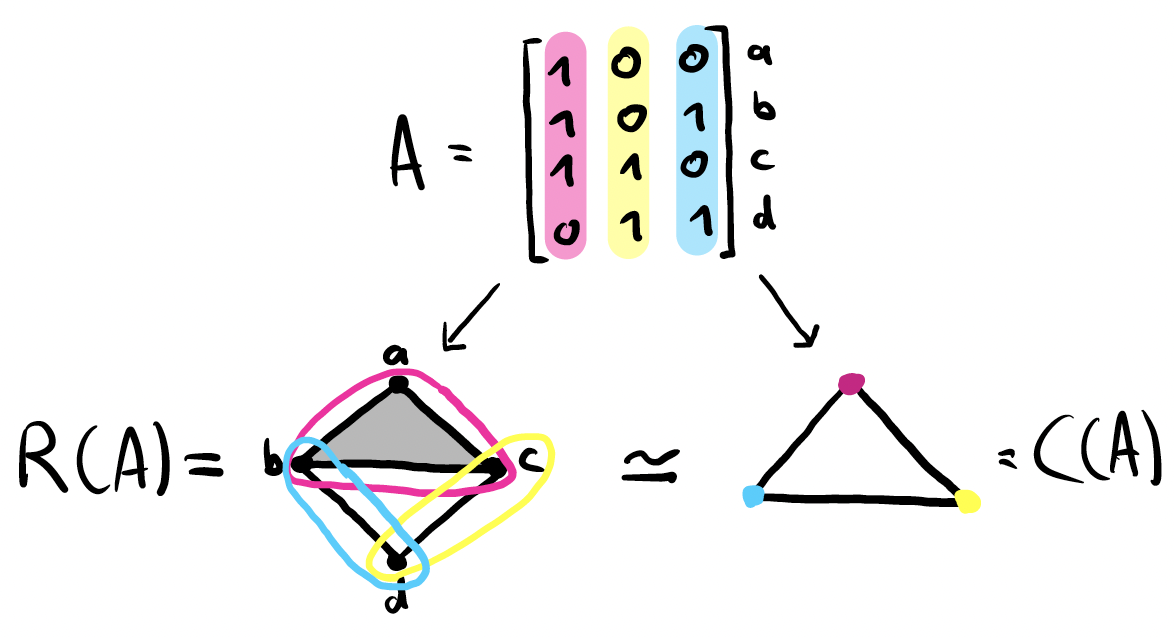}
    \label{Fig:DowkerProof}
    \caption{Construct a good open cover~$\mathcal{U}$ of the row complex~$|R(A)|$, such that the \v{C}ech nerve~$\cechnerve \mathcal{U}$ is precisely the column complex~$C(A)$. By an application of the nerve lemma we get~$|R(A)| \simeq |\cechnerve \mathcal{U}| = |C(A)|$.}
\end{figure}
Given a relation~$A \from I \times J \to \{0,1\}$ one may cover the realization~$|R(A)|$ with opens~$U_j$, obtained for every~$j \in J$ as slight thickenings of the maximal simplex~$\{i \in I \mid A(i,j)=1\}$. The nerve of this good cover is clearly equal to the relations column complex~$C(A)$ and an application of the nerve lemma thus finishes the proof. This bears the question if we can use the enriched version of the nerve lemma~$\cspace X_{\mathcal{U}} \simeq X$ as proposed by Segal, Dugger and Isaksen and prove an appropriately enriched version of Dowker duality. Specifically we study two \textit{composeable} relations~$A \from I \times J \to \{0,1\}$ and~$B \from J \times K \to \{0,1\}$. We may then augment every simplex~$\sigma$ in the row complex of~$A$ with that part of the row complex of~$B$, corresponding to~$\sigma$'s witnesses. Explicitly we may form the topological category 
\[
    \catname{R}(A,B) =
    \left[
\begin{tikzcd}
 \bigsqcup_{\sigma \in R(A)} |R(B|_{J_{\sigma}\times K})| & \bigsqcup_{\sigma_1 \subseteq \sigma_0} |R(B|_{J_{\sigma_0}\times K})| \arrow[l,shift left,"\operatorname*{id}"] \arrow[l,shift right,"\iota"']  \arrow[l,shift right]
\end{tikzcd}\right],
\]
where~$J_{\sigma}=\{j \in J|A(i,j)=1\}$ and as for~$X_{\mathcal{U}}$ above, the source and target maps are constructed from inclusion and identity maps. Similarly we may augment simplices in the column complex of~$B$ with bits of the column complex of~$A$ and form~$\catname{C}(A,B)$. An enriched version of Dowker duality would then amount to a weak equivalence between the respective classifying spaces~$\operatorname{B}\catname{R}(A,B) \simeq \operatorname{B}\catname{C}(B,A)$. A sketch of our proof for such a result is shown in Figure \ref{Fig:ExtDowkerProof}.
\begin{figure}
    \centering    \includegraphics[width=0.8\textwidth,height=0.8\textheight,keepaspectratio]{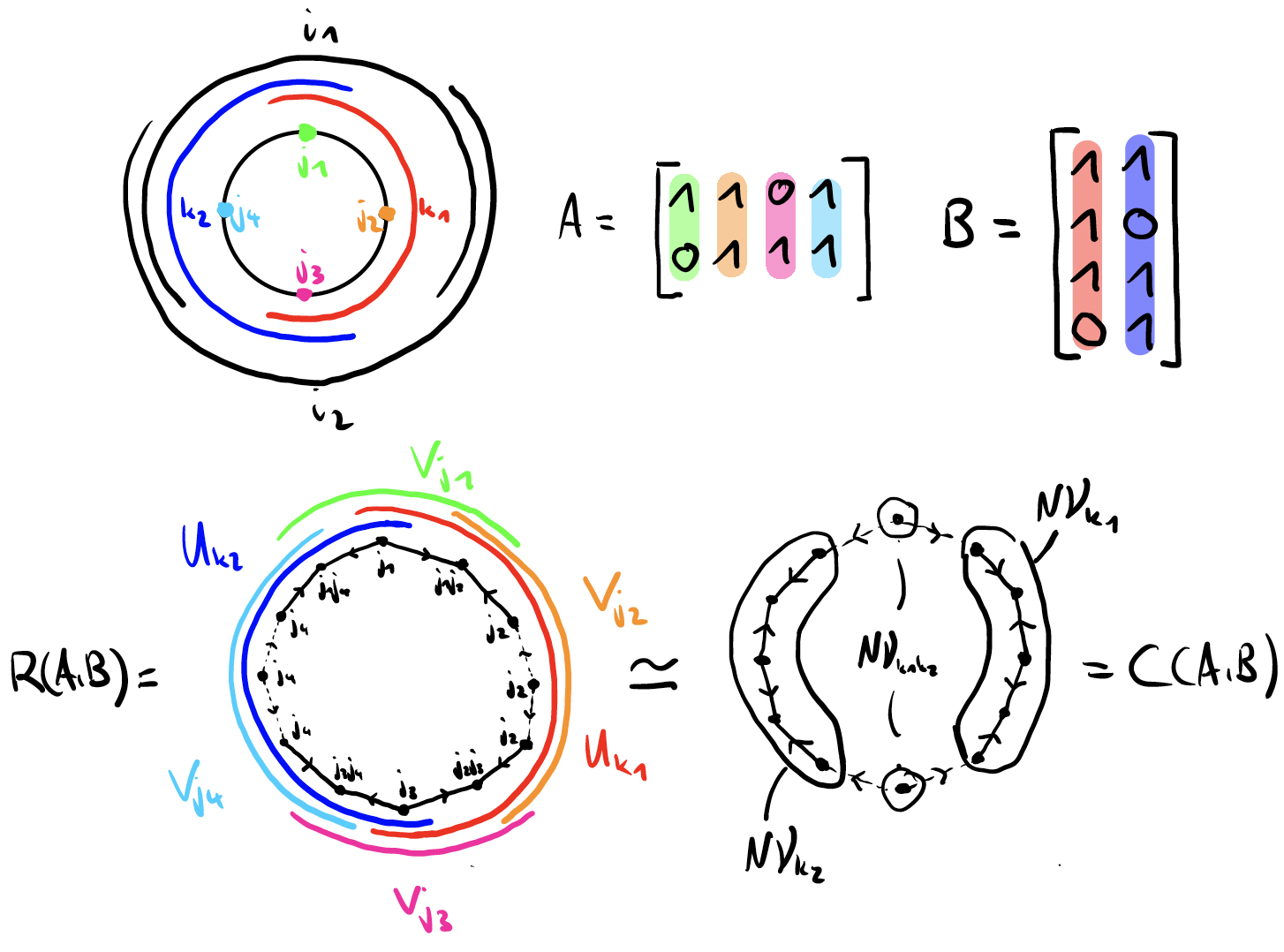}
    \label{Fig:ExtDowkerProof}
    \caption{Given two relations~$A$ and~$B$ of compatible dimensions we construct topological categories~$R(A,B)$ and~$C(A,B)$ incorporating information of the respective row and column complexes into one another. Constructing an appropriate cover~$\mathcal{U}$ of~$|R(A,B)|$ and then a good refinement of every overlap~$U_{\rho}$ in~$\mathcal{U}$ in terms of another cover~$\mathcal{V}$ we are able to reconstruct the homotopy type of~$|C(A,B)|$ and prove ~$|R(A,B| \simeq \operatorname{B}|R(A,B)|_{\mathcal{U}} \simeq |C(A,B)|$.}
\end{figure}

\paragraph*{Motivation 2} 
Experiments in neuroscience often produce datasets where the activity of~$N$ neurons is recorded during~$T$ timebins. This information may be represented as a binary matrix~$A$, where~$A_{i,j}=1$ if the~$i$'th neuron is active in the~$j$'th timebin and~$A_{i,j}=0$ if it is inactive. Our experience with these datasets shows that for certain collections of neurons it is possible to identify a covariate space~$X$ such that individual neurons have highly elevated activity levels in spatially constrained regions of~$X$. These regions are called the neurons \textit{receptive fields}. The space~$X$ might be something with an immediate analog in the real world like, for example, head direction or the position in a 2D-plane. If the covariate space~$X$ is unknown it might be of interest to infer it's homotopical invariants directly from the datamatrix~$A$. To do so, we may interpret the neurons receptive fields as elements of a good open cover~$\mathcal{U}$ of~$X$, while timebins corrrespond to points of~$X$. The matrix~$A$ is then of the type~$A_{\mathcal{U}}$ and we may infer the covers \v{C}ech nerve as it's row complex.
\begin{figure}
    \centering    \includegraphics[width=0.8\textwidth,height=0.8\textheight,keepaspectratio]{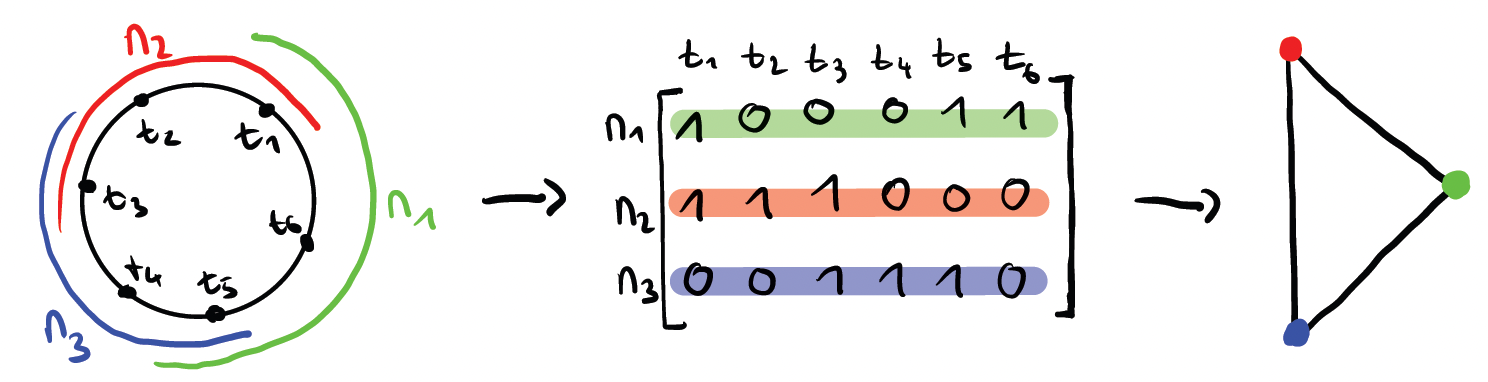}
    \label{Fig:Neuro}
    \caption{The activity of three neurons~$n_1$, $n_2$ and $n_3$ is recorded during six time points~$t_1,t_2,\ldots$ giving us a relation. The neurons receptive fields form a good cover of the circular covariate space. If the timepoints constitute a sufficient sample of that space we may compute the covers nerve as the relation's row complex.}
\end{figure}
This method for unsupervised inferrence of neural covariate spaces was pioneered by Curto, Itzkov \cite{curto2008cell}, Singh et al. \cite{singh2008topological} and others. It has since then been successfully used for example in the analysis of data recorded from head direction cells \cite{rybakken2019decoding} and grid cells \cite{gardner2022toroidal}. Such applications however don't come without challenges. While receptive fields are often convex, we may also encounter situations where they are for example multi-peaked, which would lead to a non-good cover and thus a false inference as the nerve lemma does not apply anymore. 
We may weight simplices in Dowker complexes by their number of witnesses. This is called the \textit{total weight} in \cite{robinson2022cosheaf}. To deal with receptive fields that don't form a good cover we study the so obtained filtration of column complexes 
\begin{equation*}
    \ldots \hookrightarrow C_l(A) \hookrightarrow C_{l+1}(A) \hookrightarrow \ldots
\end{equation*}
and often recover the correct homotopy type of~$X$ over a significant range of parameters. The reasons is that homotopical content of non-contractible overlaps is recovered through other receptive fields. Note that a similar weighting on the row complex does not achieve the same objective. The filtration 
\begin{equation*}
    \ldots \hookrightarrow R_k(A) \hookrightarrow R_{k+1}(A) \hookrightarrow \ldots
\end{equation*}
is however useful to deal with another potential challenge. A neural recording may contains multiple groups of neurons, with receptive fields on different spaces~$X_1,X_2,\ldots$. Then a timepoint would not correspond to a single point in one of them, but rather to a tuple~$x_1,x_2,\ldots$ i.e a point in the product space~$X_1 \times X_2\times \ldots$. The filtration of row complexes may disentangle them in a significant range of parameters and recover the homotopy type of their union~$X_1 \cup X_2 \cup \ldots$ instead. Summarizing the above, the row complex~$R(A)$ and the column complex~$C(A)$ are always homotopy equivalent but their filtrations in terms of total weights may contain complementary information. Both pieces of information can be instrumental in real world data analysis problems. We apply them to aforementioned challenges of neural data analysis in \cite{TDAneuro}. In this paper we propose a construction that coherently combines them in one single bifiltration. We hope our construction to be useful in the unsupervised inference of neural covariate spaces from datasets with multiple neural modules and potentially non-convex receptive fields. It may also be applicable to other data analysis scenarios as for example the study of gene expression profiles.  
\begin{figure}
    \centering    \includegraphics[width=0.8\textwidth,height=0.8\textheight,keepaspectratio]{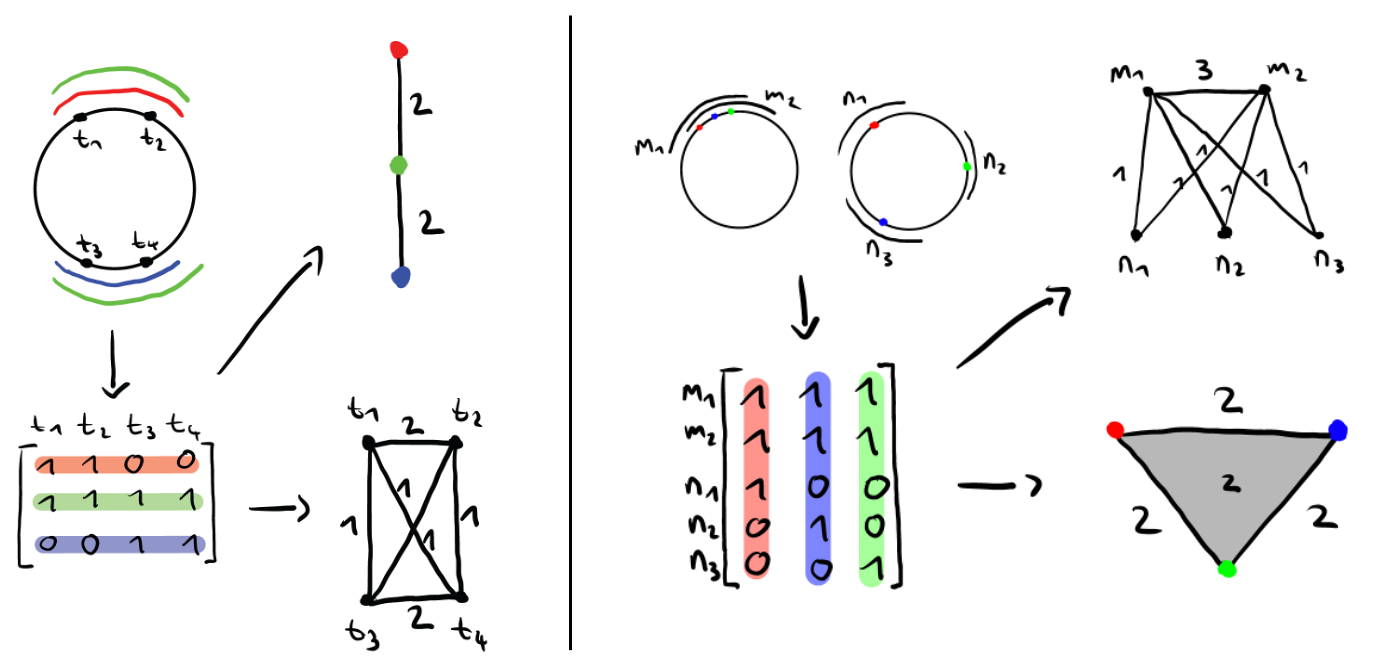}
    \label{Fig:Application}
    \caption{\textbf{Left}: For a relation~$A_{\mathcal{U}}$ coming from a non-good cover~$\mathcal{U}$ we may resolve non-contractible overlaps by considering total weights on the column complex. This does not work with the row complex. \textbf{Right}: On the other hand we may disentangle simultaneously sampled spaces through total weights on the row complex but not with the column complex.}
\end{figure}

\paragraph*{Contributions and outline} 
For formalizing our results we found it helpful to work with poset categories of the column and row complexes and their categorical nerves. A major advantage is that we can use the powerful machinery of combinatorial homotopy theory and don't have to realize for making statements about homotopy equivalences. This for example spares us the construction of thickenings and a discussion of their compatibility's in the proof of Proposition \ref{Proposition:ExtendedDowker}. On the other hand we don't sacrifice generality as we can easily translate our results back to topological spaces by applying a realization functor (see Remark \ref{remark:ssettop}). \\
In Section \ref{section:NervesSimpCat} we review some notions from combinatorial homotopy theory. First we look into some results about nerves of categories before we transition to an enriched setting of simplicial categories and associated bisimplicial sets. An emphasis is put on results that allow us to make homotopical statements about nerves of categories and their maps, based of categorical observations. Subsequent sections will make extensive use of those techniques. The reader comfortable with combinatorial homotopy theory may skip this section.\\
The purpose of section \ref{section:simpsegal} is to prove Theorem \ref{theorem:functorialsegal}. This is an enriched nerve lemma comparable to that of Segal, Dugger and Isaksen but formulated for appropriately defined covers of simplicial sets. This result is not surprising and may very well be implicit in the existing literature on the topic. We couldn't however find it in this precise form and it forms an important stepping stone for the subsequent sections.   
In Section \ref{section:ExtendedDowker} we start with two relations as in Motivation 1 and combine them to construct simplicial categories~$\catname{R}^{AB}$ and~$\catname{C}^{AB}$. These have classifying spaces weakly equivalent to those of the topological categories~$\catname{R}(A,B)$ and~$\catname{C}(A,B)$ from above. We can use our enriched nerve lemma to prove the following extension of Dowker's duality theorem. 
\begin{theorem} \label{Theorem:main}
    Given two relations~$A\from I \times J \to \{0,1\}$ and~$B\from J \times K \to \{0,1\}$ there is a zigzag of weak equivalences 
    \[
    \cspace \catname{R}(A,B) \simeq \cspace \catname{C}(A,B),
    \]
    natural with respect to morphisms of relations. 
\end{theorem}
Finally, we construct in Section \ref{section:BifiltDow} for a given relation~$A$, a bifiltration that combines the total weights on the row the column complex of~$A$ in one single object. Explicitly we construct a parameterized family of relations~$A_{kl}$, that induces a bifiltration of Dowker complexes 
\begin{equation*}
        \begin{tikzcd}
            & \vdots \arrow[d,hookrightarrow] & \vdots \arrow[d,hookrightarrow] & \\
             \cdots \arrow[r,hookrightarrow] & R(A_{k,l}) \arrow[r,hookrightarrow] \arrow[d,hookrightarrow] & R(A_{k',l}) \arrow[d,hookrightarrow] \arrow[r,hookrightarrow] & \cdots \\
            \cdots \arrow[r,hookrightarrow] & R(A_{k,l'}) \arrow[r,hookrightarrow] \arrow[d,hookrightarrow] & R(A_{k',l'}) \arrow[d,hookrightarrow] \arrow[r,hookrightarrow] & \cdots \\
            & \vdots & \vdots & 
        \end{tikzcd}
    \end{equation*} 
such that the~$(k=1)$-column recovers the filtration
\begin{equation*}
    \ldots \hookrightarrow C_l(A) \hookrightarrow C_{l+1}(A) \hookrightarrow \ldots
\end{equation*}
and the $(l=1)$-row recovers the filtration 
\begin{equation*}
    \ldots \hookrightarrow R_k(A) \hookrightarrow R_{k+1}(A) \hookrightarrow \ldots
\end{equation*}
up to natural weak equivalence. With an eye towards our motivation from neural data analysis (Motivation 2) we prove the following reconstruction result. 
\begin{theorem}
    Let~$1 \leq m \leq n$ and~$1 \leq p \leq q$. Let~$\mathcal{U} = \{U_i \subseteq X\}_{i \in I}$ be an~$n$-fold open cover and~$\mathcal{V} = \{V_j \subseteq X\}_{j \subseteq J}$ a collection of subsets of the space~$X$ such that
    \begin{enumerate}
        \item For every~$\sigma \subseteq I$ with~$\# \sigma \geq m$ the intersection~$U_{\sigma}$ is empty or contractible
        \item For every~$\sigma \subseteq I$ we have~$\#\{j \in J|U_{\sigma} \cap V_j\}\leq p$ if $U_{\sigma} = \emptyset$ and~$\#\{j \in J|U_{\sigma} \cap V_j\}\geq q$ if~$U_{\sigma} \neq \emptyset$.
    \end{enumerate}
    Then with $A \from I \times J \to \{0,1\}$ a relation where~$A(i,j)=1$ if~$U_i \cap V_j \neq \emptyset$ and zero otherwise we have for~$p \leq k \leq q$ and~$m \leq l \leq n$:
    \[
        |R(A_{k,l})| \simeq X
    \]
\end{theorem}
Computing~$R(A_{k,l})$ requires to consider all subsets~$\{\sigma \subseteq I|\# \sigma \geq l\}$ and can thus be extremely expensive from a computational point of view. With an application of our Theorem \ref{Theorem:main} we are able to prove the following simplification, which we hope facilitates the usefulness for applications.
\begin{theorem}
    For every relation~$A \from I \times J \to \{0,1\}$ and natural numbers~$k \geq k'$ and~$l\geq l'$ there are zigzags of weak equivalences that make the diagram
\[
    \begin{tikzcd}
    \left|R(A_{k,l})\right| \arrow[d] \arrow[r,dash,"\sim"] & \cspace \catname{R}(A_{1,l},A^{\top}_{k,l}) \arrow[d] \\
    \left|R(A_{k',l'})\right| \arrow[r,dash,"\sim"] & \cspace \catname{R}(A_{1,l'}A^{\top}_{k',l'}) 
    \end{tikzcd}
\]
    commute. 
\end{theorem}

\section{Nerves and simplicial categories} \label{section:NervesSimpCat}
Simplicial sets conceptually sit between simplicial complexes and topological spaces. Just like simplicial complexes they are combinatorial in nature. One way to define them is as contravariant functors~$X\from \Delta \op \to \Set$ from the \textit{simplex category}~$\Delta$ into the category of sets. This amounts to specifying for every~$n \geq 0$ a set of~$n$-simplices~$X([n])$ as well as maps between those sets. The latter amount to face and degeneracy relations and have to satisfy certain consistency conditions, as encoded in the morphisms of~$\Delta$. Simplicial sets can be thought of as directed simplicial complexes with slightly more structure and flexibility. Simplices are no longer uniquely determined by vertices and there can be~$n$-simplices between less than~$n+1$ vertices. This additional structure allows to define a homotopy theory for simplicial sets where we can talk about things like for example homotopy equivalence of simplicial maps or (weak) homotopy equivalence of simplicial sets.  This homotopy theory turns out to be equivalent to that of topological spaces. In particular there is a \textit{realization} functor that turns a simplicial set~$X$ into a topological space~$|X|$ and (by definition) a weak homotopy equivalence between simplicial sets into one between topological spaces. Excellent resources to read up on these notions from combinatorial homotopy theory are \cite{friedman2008elementary} and \cite{goerss2009simplicial}. 
\paragraph{Nerves} Given a category $\catname{C}$ we may form a simplicial set~$\nerve \catname{C}$, called it's \textit{nerve}, that with it's~$n$-simplices summarizes the possible $n$-fold compositions in~$\catname{C}$. Thus an element of~$\nerve \catname{C}([2])$ is a triangular diagram 
\begin{equation*}
    \begin{tikzcd}
    & B \arrow[rd,"g"] &  \\
    A \arrow[ru,"f"] \arrow[rr,"g\circ f"] & &  C,
    \end{tikzcd}.
    \end{equation*} 
while an element of~$\nerve \catname{C}([n])$ may be depicted as
\begin{equation*}
		\begin{tikzcd}[row sep=scriptsize, column sep=scriptsize]
		& B \arrow[rrdd,"g"] \arrow[dd,"h\circ g"] & &  \\ & & & & \\
		& D & & C \arrow[ll,"h"]\\
		A \arrow[ruuu,bend left=8,"f"] \arrow[rrru,bend right=18,"g\circ f"'] \arrow[ru,"h\circ g \circ f"'] & & & & 
		\end{tikzcd}.
\end{equation*}
 Extending the nerve construction to a functor 
\begin{align*}
    \nerve : \Cat \rightarrow \sSet
\end{align*}
is straightforward. Here we denote by~$\Cat$ the category of categories and by~$\sSet$ the category of simplicial sets. Taking the realization of a category's nerve yields a topological space called that it's \textit{classifying space}, which we denote~$\operatorname{B}\catname{C} = |\nerve \catname{C}|$. Constructing simplicial sets as nerves of categories gives us powerful tools for making statements about their homotopical behaviour based of categorical observations.
\\ For instance, given two functors~$F,G : \catname{C} \rightarrow \catname{D}$ between categories~$\catname{C}$ and~$\catname{D}$, we may interpret a natural transformation~$\eta : F \Rightarrow G$ as a functor~$\eta : \catname{C} \times (0 \rightarrow 1) \rightarrow \catname{D}$. Here we denote by~$(0 \rightarrow 1)$ the category with two objects and one non-trivial morphism between those. Using that the nerve operation is a right adjoint and thus preserves products we immediately get the following. 
\begin{lemma}\label{lemma:NatTrafoHomotopy}
Taking the nerve of two functors~$F,G : \catname{C} \rightarrow \catname{D}$, connected by the natural transformation~$\eta : F \Rightarrow G$, yields homotopic simplicial maps~$\nerve F \sim \nerve G$.
\end{lemma}
An object~$t$ in the category~$\catname{C}$ is called terminal if it exists for every object~$c$ of~$\catname{C}$ a unique morphism~$c \rightarrow t$. This is equivalent to characterising~$t$ as the colimit over the empty diagram in~$\catname{C}$.
\begin{lemma}
    The nerve of a category with a terminal object is contractible. 
\end{lemma}
\begin{proof}
    Let~$\catname{C}$ be such a category. Denote by~$(0)$ the category with just one object and the respective identity morphism. There is a unique functor~$F \from \catname{C} \to (0)$ and we define~$G \from (0) \to \catname{C}$ by mapping to the terminal object in~$\catname{C}$. Using the universal property of terminal objects we obtain a natural transformation~$G \circ F  \Rightarrow \operatorname{id}$, which by Lemma (\ref{lemma:NatTrafoHomotopy}) exposes~$\nerve F$ to be a deformation retract.   
\end{proof}
Dually, an intitial object~$i$ in a category~$\catname{C}$ admits for every object~$c \in \ob \catname{C}$ a unique morphism~$i \rightarrow c$ and is characterised as the limit over the empty diagram in~$\catname{C}$. Applying a similar strategy as above we prove the dual statement. 
\begin{lemma}
    The nerve of a category with an initial object is contractible. 
\end{lemma}
Given a functor~$F : \catname{C} \rightarrow \catname{D}$ we can form for every object~$d\in \ob \catname{D}$ it's fiber~$F/d$. This is a category with objects given by pairs~$(c\in \ob \catname{C},\alpha : F(c) \rightarrow d)$. It's importance for us is due to a result, proven in \cite{Quillen} and known as \textit{Quillens Theorem A}, that is useful for identifying functors which become weak equivalences upon applying the nerve.
\begin{theorem}\label{QuillenA}
    If for every~$d\in \ob \catname{D}$, the nerve of the fiber~$F/d$ is contractible, then~$\nerve F$ is a weak equivalence.   
\end{theorem}
\paragraph{Simplicial categories and bisimplicial sets}
Viewing simplicial sets as contravariant functors from the simplex category~$\Delta$ into the category~$\catname{Set}$ of sets suggests to also consider functors with the same domain but different target categories. Such functors are called \textit{simplicial objects in} the respective target category. A first example is simplicial objects in~$\catname{Cat}$. Given such a functor into the category of categories~$\catname{C} \from \Delta \op \to \catname{Cat}$ we may compose it with the nerve functor from above, giving us
\[
    (\nerve \circ \catname{C})  = \Delta \op \overset{\catname{C}}{\rightarrow} \catname{Cat} \overset{\nerve}{\rightarrow} \sSet.
\]
This is an instance of a simplicial object in~$\sSet$ also called a \textit{bisimplicial set}. Explaining the naming convention, we can equivalently view a bisimplcial set~$X \from \Delta \op \to \sSet$ as a functor~$\Delta \op \times \Delta \op \to \Set$. There are several ways to turn a bisimplicial set back into a mere simplicial set but all of these are naturally isomorphic. Often it is most convenient to take the \textit{diagonal} that sends~$X \from \Delta \op \rightarrow \sSet$ to~$\diag X \from \Delta \op \to \Set$ with~$\diag X([n]) = X([n],[n])$.

A very useful property of the diagonal functor is that it preserves weak equivalences in the following sense. 
\begin{lemma}\label{lemma:SimpCatLevelwiseWeakEq}
    If~$F : X \rightarrow Y$ is a map of bisimplicial sets, such that for every~$n \geq 0$ the component~$F_{[n]} \from X([n]) \to Y([n])$ is a weak equivalence, the induced map~$\diag F : \diag X \rightarrow \diag Y$ is a weak equivalence as well.  
\end{lemma}
This is proven as Proposition 1.9 of Chapter 4 in \cite{goerss2009simplicial}. \\
A simplicial object in~$\Cat$ may equivalently be viewed as a category internal to simplicial sets. This is a category with a simplicial set of objects~$C_0$, a simplicial set of morphisms~$C_1$ and simplicial structure maps of which we may suppress the identity assigning map and the composition map to display
\begin{equation*}
        \begin{tikzcd}
        C_0 & C_1 \arrow[l,shift right,"s"'] \arrow[l,shift left,"t"]
        \end{tikzcd}.
    \end{equation*} 
This information is equivalently encoded in a functor~$\catname{C} \from \Delta \op \to \sSet$ with
\[
    \catname{C}([n]) = \begin{tikzcd}
        C_0([n]) & C_1([n]) \arrow[l,shift right,"s_n"'] \arrow[l,shift left,"t_n"]
        \end{tikzcd}.
\]
There is a generalization of Quillens theorem A to functors between simplicial categories, or equivalently, maps of simplicial objects in $\Cat$. Given such a functor~$F : \catname{C} \rightarrow \catname{D}$ we construct for every~$n \geq 0$ and object~$d \in \ob \catname{D}([n])$ a simplicial object in~$\Cat$ denoted by $F/d$ and with category of~$m$-simplices given by 
\begin{equation*}
    \bigsqcup_{\theta : [m] \rightarrow [n]} F_{[m]}/\theta^*d
\end{equation*}
In \cite{Waldhausen} it is proven that
\begin{theorem}\label{QuillenA'}
    If for every~$d \in \ob \catname{D}([n])$ the simplicial category~$\diag \nerve \circ F/d$ is contractible, then $F$ induces a weak homotopy equivalence $\diag (\nerve \circ \catname{C}) \overset{\sim}{\to} \diag (\nerve \circ \catname{D})$.   
\end{theorem}
\paragraph{Topological categories}
\textit{Topological categories} have topological spaces of objects, of morphisms and continuous structure maps. There are different ways to turn a topological category into a topological space \cite{segal1974categories}. Here we will use the \textit{geometric realization} and denote, for a topological category~$\catname{T}$, the resulting space with~$\operatorname{B} \catname{T}$ ,thereby overloading the symbol $\cspace$, which we also used for the classifying space of a mere category. \\
Starting with a simplicial category
\[ \catname{C} = 
\begin{tikzcd}
    C_0 & C_1 \arrow[l,shift left,"t"] \arrow[l,shift right,"s"'] 
\end{tikzcd},
\]
we can realize to a topological category as 
\[
|\catname{C}| = 
\begin{tikzcd}
    \left|C_0\right| & \arrow[l,shift left,"|t|"] \arrow[l,shift right,"|s|"'] |C_1|,
\end{tikzcd}
\]
which we may turn into a topological space~$\operatorname{B}|\catname{C}|$. As already discussed above we can equivalently view a simplicial category~$\catname{C}$ as a functor~$\Delta \op \to \catname{Cat}$, turn it into a simplicial set~$\diag (\nerve \circ \catname{C})$ and realize to~$|\diag (\nerve \circ \catname{C})|$. There is a natural homeomorphism~$|\diag \nerve \circ \catname{C}| \simeq \operatorname{B}|\catname{C}|$. 
\\
A simplicial space is a simplicial object in the category of topological spaces i.e a functor~$\catname{T} \from \Delta \op \to \catname{Top}$. 
Different from bisimplicial sets, not all levelwise weak equivalences between simplicial spaces realize to weak equivalences of topological spaces. A simplicial space is called~$\textit{good}$ if all degeneracy maps are closed cofibrations. This is for example true if it is levelwise a CW-complex. The following is proven in \cite{segal1974categories}.
\begin{proposition}\label{proposition:DegreewiseWeakTop}
    Let~$F \from \catname{C} \to \catname{D}$ be a map between good simplicial spaces that is degreewise a weak equivalence. Then the induced map~$\operatorname{B}F \from \operatorname{B}\catname{C} \to \operatorname{B}\catname{D}$ is also a weak equivalence.  
\end{proposition}

\section{A simplicial Segal nerve lemma} \label{section:simpsegal}
In the introduction we discussed how to associate a topological category to an open cover of a topological space~$X$ such that it's classifying space has the homotopy type of~$X$. We will now give a similar result for the case where~$X$ is a simplicial set and we form a simplicial category~$X_{\mathcal{U}}$ from a \textit{simplicial cover}~$\mathcal{U}$ of~$X$. \\
First we should make precise what we mean by simplicial cover. Call~$Y$ a \textit{sub-simplicial set} of~$X$ if there is a simplicial map~$Y \overset{\iota}{\hookrightarrow} X$ such that for all~$n \geq 0$ the map of sets~$Y([n]) \overset{\iota_{[n]}}{\to} X([n])$ is an inclusion. We call a family~$U=\{U_i \overset{\iota_i}{\hookrightarrow} X\}_{i \in I}$ a simplicial cover of~$I$ if for all~$n \geq 0$ we have~$\bigcup_{i \in I} U_i([n]) = X([n])$. Let us denote by $U_{\sigma}$ the intersection~$\bigcap_{\alpha \in \sigma} U_{\alpha}$, formally given by the limit of the diagram made up of all the relevant inlusion maps~$U_i \overset{\iota_i}{\hookrightarrow} X$. We may define a poset category~$\catname{U}$ with objects given as
\[
    \ob \catname{U} = \{\sigma \subseteq I | U_{\sigma} \neq \emptyset\}.
\]
The nerve of this category~$\nerve \catname{U}$ is a simplicial set that we can intuitively think of as (a subdivision of) the \v{C}ech nerve of the cover~$\mathcal{U}$. We may augment~$\catname{U}$ to a simplicial category
\begin{equation*} X_{\mathcal{U}} = 
    \begin{tikzcd}
        \bigsqcup_{\sigma \subseteq I} U_{\sigma} & \bigsqcup_{\sigma' \subseteq \sigma} U_{\sigma} \arrow[l,shift right] \arrow[l,shift left]
    \end{tikzcd}.
\end{equation*}
The source and target maps are given by identity, mapping~$x \in U_{\sigma}([n])$ from the component~$\sigma' \subseteq \sigma$ to~$x \in U_{\sigma}$ and inclusion, mapping it to~$x \in U_{\sigma'}$. This is supposed to directly mimic the construction of a topological category from an open cover of a topological space as in \cite{Segal}. 
\begin{proposition} \label{SimplicialSegal}
    The natural map~$\diag (\nerve \circ X_{\mathcal{U}}) \rightarrow X$ is a weak homotopy equivalence.  
\end{proposition}
\begin{proof}
    We apply the simplicial version of Quillens theorem A (Theorem \ref{QuillenA'}). To do so we interpret~$X$ as a simplicial category in the trivial way i.e for every~$n$ the category~$X([n])$ just has identity morphisms. We then define the functor~$F : X_U \rightarrow X$ as sending~$x \in U_{\sigma}$ to~$x \in X$. We must check that for every~$n \geq 0$ and every~$x \in X([n])$, the simplicial category~$F/x$ is contractible. This follows because for every~$m\geq 0$ the category 
    \begin{equation*}
        (F/x)_{[m]} = \bigsqcup_{\theta : [m] \rightarrow [n]} F_{[m]}/\theta^*x
    \end{equation*}
    is contractible, as~$F_{[m]}/\theta^*x$ has a terminal object given by~$(\theta^*x)_{\sigma}$ for~$\sigma=\{i \in I | \theta^*x \in U_{i}([m])\}$. 
\end{proof}
Let us return for a moment to the category~$\catname{U}$ from above. We may interpret it as a simplicial category in the trivial way i.e
    \[
        \catname{U} =
            \begin{tikzcd}
                \bigsqcup_{\sigma \subseteq I} \ast & \bigsqcup_{\sigma' \subseteq \sigma} \ast \arrow[l,shift right] \arrow[l,shift left]
            \end{tikzcd}
    \] 
and get an obvious simplicial functor~$X_{\mathcal{U}} \to \catname{U}$. 
For a cover~$\{U_i \hookrightarrow X\}_{i \in I}$, where all intersections~$U_{\sigma}$ are either empty or contractible we have as a direct consequence of Lemma \ref{lemma:SimpCatLevelwiseWeakEq} that this map induces a homotopy equivalence
    \[
        \diag (\nerve \circ X_U) \simeq \nerve \catname{U}
    \]
Combining this observation with Proposition \ref{SimplicialSegal} we obtain a simplicial set version of the nerve lemma. 
\begin{corollary}\label{simplicialnerve}
    Let~$\mathcal{U}=\{U_i \hookrightarrow X\}_{i \in I}$ be a cover of the simplicial set~$X$, such that all intersections~$U_{\sigma}$ as defined above are either empty or contractible. Then
    \[
        \nerve \catname{U} \simeq X
    \]  
\end{corollary}
As for open covers of topological spaces we may from now on also call simplicial covers with contractible or empty intersections \textit{good} covers.\\
Now let~$f : X \rightarrow Y$ be a simplicial map and~$\mathcal{U}=\{ U_i \hookrightarrow X\}_{i \in I}$ and~$\mathcal{V}=\{V_j \hookrightarrow Y\}_{j \in J}$ covers such that for every~$i \in I$ there is some~$p(i) \in J$ so that for every~$k \geq 0$ we have~$f_{[k]}U_i([k]) \subseteq V_{p(i)}([k])$. We may fix such a choice for a pairing i.e an appropriate map~$p : I \rightarrow J$ and obtain an induced map of simplicial categories 
\begin{equation*}
    \tilde{f}^p : X_{\mathcal{U}} \rightarrow Y_{\mathcal{V}}
\end{equation*}
that sends~$x \in U_{\sigma}([k])$ to $f_{[k]}(x) \in V_{p(\sigma)}([k])$. Note that for two choices of pairing~$p_0 : I \rightarrow J$ and $p_1: I \rightarrow J$ we obtain different but homotopic maps. This follows from an application of Lemma \ref{lemma:NatTrafoHomotopy} to the zigzag of natural transformations into~$\tilde{f}^{p_0p_1} \from X_{\mathcal{U}} \to Y_{\mathcal{V}}$, where~$\tilde{f}^{p_0p_1}_{[k]}(x \in U_{\sigma}([k])=(f_{[k]}(x) \in V_{p_0(\sigma)p_1(\sigma)})$. We will from now on omit mentioning the choice of pairing and just denote the induced map on simplicial categories by~$\tilde{f} : X_U \rightarrow Y_V$. Proposition \ref{SimplicialSegal} may then be strengthened to the following functorial formulation. 
\begin{theorem}\label{theorem:functorialsegal}
    Given two covers of simplicial sets~$\mathcal{U}=\{ U_{\alpha} \hookrightarrow X\}_{\alpha \in I}$ and~$\mathcal{V}=\{V_{\beta} \hookrightarrow Y\}_{\beta \in J}$ and a simplicial map~$f : X \rightarrow Y$, compatible with the covers~$\mathcal{U}$ and $\mathcal{V}$ as above, the diagram
    \begin{equation*}
        \begin{tikzcd}
            \nerve \operatorname{U} \arrow[d] & \diag (\nerve \circ X_{\mathcal{U}}) \arrow[d,"\diag \nerve \tilde{f}"] \arrow[r] \arrow[l] & X \arrow[d,"f"] \\ 
            \nerve \operatorname{V} & \diag (\nerve \circ Y_{\mathcal{V}}) \arrow[r] \arrow[l] & Y
        \end{tikzcd}
    \end{equation*}
    commutes and the vertical maps are weak equivalences. 
\end{theorem}

\section{Extended Dowker duality} \label{section:ExtendedDowker}
\paragraph{Ordinary Dowker duality} Let $A \from I \times J \to \{0,1\}$ be a relation. For every~$\sigma \subseteq I$ we may define the set of elements satisfying~$A$ with all elements of~$\sigma$:
\[
    J^A_{\sigma} = \{j \in J \mid \forall i \in \sigma A(i,j)=1\}.
\] 
Similarly for every $\tau \subseteq J$ we get
\[ 
    I^A_{\tau} = \{i \in I | \forall j \in \tau :  A(i,j)=1\}.
\]
Then we may define two poset categories~$\catname{R}^A$ and~$\catname{C}^A$ with sets of objects given as 
\begin{align*}
    & \ob \catname{R}^A = \{\sigma \subset I \mid J^A_{\sigma} \neq \emptyset\} \\
    & \ob \catname{C}^A = \{\tau \subset J \mid I^A_{\tau} \neq \emptyset\}.
\end{align*}
These are the poset categories of the row and column complexes of~$A$ and consequently their classifying spaces are homotopy equivalent to the respective realizations~$\cspace \catname{R}^A \simeq |R(A)|$ and~$\cspace \catname{C}^A \simeq |C(A)|$. Let us call~$\catname{R}^A$ and~$\catname{C}^A$ the relations \textit{row and column categories}.  As a warmup we reprove Dowker duality. 
\begin{theorem}
    The simplicial sets~$R^A= \nerve \catname{R}^A$ and~$C^A = \nerve \catname{C}^A$ are weakly equivalent. 
\end{theorem}
\begin{proof}
    For every~$\tau \subseteq J$ we may form the full subcategory~$\catname{R}^A_{I_{\tau}}$ of~$\catname{R}^A$ with objects 
    \[
        \ob \catname{R}_{I_{\tau}}^A = \{\sigma \subset I_{\tau} \mid J^A_{\sigma} \neq \emptyset\}.
    \]
    Either~$I_{\tau}$ is empty and thus the subcategory~$\catname{R}^A_{I_{\tau}}$ is as well, or it is a terminal object. Thus all overlaps in the simplicial cover given as
    \[
    \mathcal{U} = \{U_j = \nerve \catname{R}^A_{I_{j}} \hookrightarrow R^A\}_{j \in J}
    \]
     are either empty or contractible. On the other hand the poset category~$\catname{U}$ with objects
     \[
       \ob \catname{U} = \{\tau \subseteq J \mid U_{\tau}\}
    \]
    is presisely~$\catname{C}^A$ and thus Corollary \ref{simplicialnerve} gives us 
    \[
        R^A \simeq \nerve \catname{U} = C^A
    \]
\end{proof}
\paragraph{Two covers}
Consider now two simplicial covers of the same space:~$\mathcal{U}=\{U_i \hookrightarrow X\}_{i \in I}$ and~$\mathcal{V}=\{V_j \hookrightarrow X\}_{j \in J}$. We define the following simplicial category 
\[ \catname{U_V} =\left[
\begin{tikzcd}
    \bigsqcup_{\sigma \in \ob \catname{U}} \nerve \catname{V}_{\sigma} & \bigsqcup_{(\sigma' \subseteq \sigma) \in \mor \catname{U}} \nerve \catname{V}_{\sigma} \arrow[l, shift right, "\iota"'] \arrow[l, shift left, "\operatorname{id}"]
\end{tikzcd}\right],
\]
where~$\catname{U}$ denotes the poset category with objects~$\{\sigma \subseteq I \mid U_{\sigma} \neq \emptyset\}$ and~$\catname{V}_{\sigma}$ the poset category with objects~$\{\tau \subseteq J \mid U_{\sigma} \cap V_{\tau} \neq \emptyset\}$.
\begin{lemma} \label{lemma:TwoCovers}
    Let~$\mathcal{U} = \{U_i \hookrightarrow X\}_{i \in I}$ and~$\mathcal{V} = \{V_j \hookrightarrow X\}_{j \in J}$ be two simplicial covers such that for all~$\sigma \subseteq I$ and~$\tau \subseteq J$ the intersection~$U_{\sigma} \cap V_{\tau}$ is either empty or contractible. Then
    \[
    \diag (\nerve \circ \catname{U_V}) \simeq X.
    \]
\end{lemma}
\begin{proof}
For every~$\sigma \subseteq I$ the collection~$\mathcal{V_{\sigma}}=
\{U_{\sigma} \cap V_j\}_{j \in J}$ is a good simplicial cover of~$U_{\sigma}$. 
We thus get a zigzag of weak equivalences 
\[
\nerve \catname{V}_{\sigma} \overset{\sim}{\leftarrow} \diag (\nerve \circ (U_{\sigma})_{\mathcal{V}_{\sigma}}) \overset{\sim}{\rightarrow} U_{\sigma}. 
\]
Applying Lemma \ref{lemma:SimpCatLevelwiseWeakEq} followed by Proposition \ref{SimplicialSegal} finishes the proof. 
\end{proof}
\paragraph{Extended Dowker duality}
Consider now two relations~$A \from I \times J \to \{0,1\}$ and~$B \from J \times K \to \{0,1\}$. We will need the following subsets defined for all~$\sigma \subseteq I$, $\tau \subseteq J$ and $\rho \subseteq K$: 
\begin{align*}
    & I^A_{\tau} = \{i \in I | \forall j \in \tau :  A(i,j)=1\}, \quad J^A_{\sigma} = \{j \in J | \forall i \in \sigma :  A(i,j)=1\}, \\
    & J^B_{\rho} = \{j \in J | \forall k \in \rho :  B(i,j)=1\}, \quad K^B_{\tau} = \{k \in K | \forall k \in \rho :  B(i,j)=1\}.
\end{align*}
For every~$\sigma \in \ob \catname{R}^A$ we may restrict the relation~$B|_{J^A_{\sigma}}$ and denote the corresponding row category by~$\catname{R}^B_{J_{\sigma}}$. Note that for~$\sigma' \subseteq \sigma$ we have~$\catname{R}^B_{J_{\sigma}} \hookrightarrow \catname{R}^B_{J_{\sigma'}}$.  We may assemble from this information a simplicial category
\begin{equation*}
    \catname{R}^{AB} = \begin{tikzcd}
        \bigsqcup_{\rho \in \ob \catname{R^A}} \nerve \catname{R}^B_{J_{\sigma}} & \bigsqcup_{(\sigma' \subseteq \sigma)\in \mor \catname{R^A}} \nerve \catname{R}^B_{J_{\sigma}} \arrow[l,shift left] \arrow[l,shift right]
    \end{tikzcd}
\end{equation*}
The source and target maps are given by the obvious inclusion and identity functor. Similarly we may form a simplicial category
\begin{equation*}
    \catname{C}^{BA} = \begin{tikzcd}
        \bigsqcup_{\rho \in \ob \catname{C^B}} \nerve \catname{C}^A_{J_{\rho}} & \bigsqcup_{(\rho' \subseteq \rho) \in \mor \catname{C^B}} \nerve \catname{C}^A_{J_{\rho}} \arrow[l,shift left] \arrow[l,shift right]
    \end{tikzcd}.
\end{equation*}
We will denote~$R^{AB} = \diag (\nerve \circ \catname{R}^{AB})$ and~$C^{BA} = \diag (\nerve \circ \catname{C}^{BA})$ 
\begin{remark} \label{remark:ssettop}
    In the introduction we stated our contributions in terms of topological categories~$\catname{R}(A,B)$ and~$\catname{C}(A,B)$. We may apply the realization functor levelwise to the here defined~$(\nerve \circ \catname{R}^{AB})$ and~$(\nerve \circ \catname{C}^{AB})$. This results in simplicial spaces, degree wise weakly equivalent to~$\catname{R}(A,B)$ and~$\catname{C}(A,B)$. Since all the spaces in question are~$CW$-complexes all the simplicial spaces are good. According to Proposition \ref{proposition:DegreewiseWeakTop} we get induced natural weak homotopy equivalences between the realizations. All the results we prove in this and the following section about the simplicial categories~$\catname{R}^{AB}$ and~$\catname{C}^{AB}$ and their realizations translate immediately to the results stated in the introduction. As a bonus we gain insight into the combinatorics which may be required for an algorithmic implementation of our methods. 
\end{remark}
\begin{proposition}\label{Proposition:ExtendedDowker}
    The simplicial sets~$R^{AB}$ and~$C^{BA}$ are weakly homotopy equivalent. 
\end{proposition}
\begin{proof}
We will define two simplicial covers~$\mathcal{U}=\{U_k \hookrightarrow R^{AB}\}_{k \in K}$ and~$\mathcal{V}=\{V_j \hookrightarrow R^{AB}\}_{j \in J}$ that satisfy the conditions of Lemma \ref{lemma:TwoCovers} and furthermore have the property that~$\diag (\nerve \circ \catname{U_V})=C^{AB}$ and thus 
\[
R^{AB} \simeq \diag (\nerve \circ \catname{U_V}) = C^{AB}.
\]
The cover~$\mathcal{U}$ is obtained by forming for every~$k \in K$ the simplicial set
\[
    U_k = \diag \left(\nerve \circ \left[\begin{tikzcd} \bigsqcup_{\sigma \in \ob \catname{R}^A} \nerve \catname{R}^B_{J^A_{\sigma} \cap J^B_k} & \bigsqcup_{(\sigma' \subseteq \sigma) \in \mor \catname{R}^A} \nerve \catname{R}^B_{J^A_{\sigma} \cap J^B_k} \arrow[l,shift left] \arrow[l, shift right] \end{tikzcd}\right]\right).
\]
Indeed this gives us a cover. If~$\tau_0 \subset \ldots \subseteq \tau_n$ is an $n$-simplex of~$\nerve \catname{R}^{B}_{J^A_{\sigma}}$ it is one in~$U_{\sigma}$ for every~$k \in K^B_{\tau_n}$. 
Next we define~$\mathcal{V}$ by forming for every~$j \in J$ a simplicial set
\[
    V_j = \diag \left(\nerve \circ \left[\begin{tikzcd}
        \bigsqcup_{\sigma \in \ob \catname{R}_{I^A_{j}}^A} \nerve \catname{R}^B_{J^A_{\sigma}} & \bigsqcup_{(\sigma' \subseteq \sigma) \in \mor \catname{R}_{I^A_{j}}^A} \nerve \catname{R}^B_{J^A_{\sigma}} \arrow[l,shift left] \arrow[l, shift right]
    \end{tikzcd} \right]\right).
\]
This is indeed a cover as~$\sigma \subseteq I_j$ for every $j \in J^A_{\sigma}$. 
For all~$\sigma \subseteq I$ and~$\tau \subseteq J$ the intersection~$U_{\sigma} \cap V_{\tau}$ is given as 
\[
    U_{\sigma} \cap V_{\tau} = \diag \left(\nerve \circ \left[\begin{tikzcd}
        \bigsqcup_{\sigma \in \ob \catname{R}_{I^A_{\tau}}^A} \nerve \catname{R}^B_{J^A_{\sigma} \cap J^B_{\rho}} & \bigsqcup_{(\sigma' \subseteq \sigma) \in \mor \catname{R}_{I^A_{\tau}}^A} \nerve \catname{R}^B_{J^A_{\sigma} \cap J^B_{\rho}} \arrow[l,shift left] \arrow[l, shift right]
    \end{tikzcd}\right]\right).
\]
This is non-empty if and only if~$\tau \subseteq J^B_{\rho}$ and~$I^A_{\tau}$ is not empty i.e if and only if~$\tau \in \ob \catname{C}^A_{J^B_{\rho}}$. Thus we have 
\[
\nerve \catname{V}_{\rho} = \nerve \catname{C}^A_{J_{\rho}}
\]
Finally, if~$U_{\sigma} \cap V_{\tau}$ is not empty, then for every~$\sigma \in I^A_{\tau}$ the category~$\catname{R}^B_{J^A_{\sigma}\cap J^B_{\rho}}$ contains the terminal object~$J^A_{\sigma}\cap J^B_{\rho}$ and is thus contractible. Thus~$U_{\rho}\cap V_{\tau} \simeq \nerve \catname{R}^A_{I^A_{\tau}}$, which is contractible as well.  
\end{proof}

We may strengthen this result to a functorial formulation. First consider a simplicial map~$f \from X \to Y$ and simplicial covers~$\mathcal{U} = \{U_i \hookrightarrow X\}_{i \in I}$,~$\mathcal{S} = \{S_i \hookrightarrow X\}_{k \in K}$,~$\mathcal{V} = \{V_j \hookrightarrow Y\}_{j \in J}$ and~$\mathcal{T} = \{T_l \hookrightarrow Y\}_{l \in L}$, such that there are maps~$\alpha \from I \to K$ and~$\beta \from J \to L$ where 
\[
f(U_{i}) \subseteq S_{\alpha(i)} \quad \text{for all }i \in I
\]
and 
\[
f(V_{j}) \subseteq T_{\beta(j)} \quad \text{for all }j \in J
\]
We then get an induced map~$\diag (\nerve \circ \catname{U_V}) \to \diag (\nerve \circ \catname{S_T})$ that does not depend on the choice of~$\alpha$ and~$\beta$ up to homotopy. 
We can augment our proof of Lemma \ref{lemma:TwoCovers} with Theorem \ref{theorem:functorialsegal}.
\begin{lemma}\label{lemma:FunctorialTwoCovers}
    Given a simplicial map~$f \from X \to Y$ and covers~$\mathcal{U,V,S,T}$ as above there is a commutative diagram
    \[
    \begin{tikzcd}
        \diag (\nerve \circ \catname{U_V}) \arrow[r,dash,"\sim"] \arrow[d] & X \arrow[d,"f"] \\
        \diag (\nerve \circ \catname{S_T}) \arrow[r,dash,"\sim"] \arrow[r,dash,"\sim"] & Y
    \end{tikzcd}
    \]
\end{lemma}
Let us now go back to relations. A morphism~$a \from A \to A'$ from the relation~$A\from I \times J \to \{0,1\}$ to the relation~$A'\from I' \times J' \to \{0,1\}$ is a pair of maps of sets~$a_0 \from I \to I'$ and~$a_1 \from J \to J'$ such that for all~$(i,j)\in I\times J$ we have that~$A(i,j)=1$ implies~$A'(a_0(i),a_1(i))=1$.
Consider now four relations 
\begin{align*}
    &A \from I \times J \to \{0,1\}, \quad A' \from I' \times J' \to \{0,1\}, \\
    &B \from J \times K \to \{0,1\}, \quad B' \from J' \times K' \to \{0,1\} \\
\end{align*}
and morphisms of relations~$a \from A \to A'$ and~$b \from B \to B'$. 
Then the obvious maps of poset categories induce simplicial maps   
\[
R^{ab} \from R^{AB} \to R^{A'B'}
\]
and 
\[
C^{ab} \from C^{AB} \to C^{A'B'}.
\]
We may prove the following.
\begin{theorem}\label{theorem:FunctorialExtendedDowker}
    For~$A$, $A'$, $B$ and $B'$ relations and~$a \from A \to A'$ and~$b \from B \to B'$ morphisms of relations as above, the induced diagram  
     \[
    \begin{tikzcd}
        R^{AB} \arrow[r,dash,"\sim"] \arrow[d,"R^{ab}"] & C^{AB} \arrow[d,"C^{ab}"] \\
        R^{A'B'} \arrow[r,dash,"\sim"] & C^{A'B'}
    \end{tikzcd}
    \]
        commutes.
    \end{theorem}
    \begin{proof}
        It is straightforward to see, that the covers in the proof of Proposition \ref{Proposition:ExtendedDowker} are compatible (in the sense discussed in Section \ref{section:simpsegal}) with the induced maps~$R^{ab}$ and~$C^{ab}$ as required for an application of Lemma \ref{lemma:FunctorialTwoCovers}.
    \end{proof}

\section*{Bifiltered Dowker complexes} \label{section:BifiltDow}
Given a relation~$A\from I \times J \to \{0,1\}$, the \textit{total weight} on~$\sigma \subseteq I$ is the cardinality of the set~$J_{\sigma}=\{j \in J \mid \forall i \in \sigma : A(i,j)=1\}$. In other words: it is the number of witnesses to the potential simplex~$\sigma$ in the row complex. Similarly we define total weights for simplices~$\tau \subseteq J$. As already discussed in \cite{robinson2022cosheaf}, one may filter the row and column complexes in terms of these total weights. In our terminology this amounts to forming subcategories for every $k,l \geq 1$ of~$\catname{R}^A$ and~$\catname{C}^A$ with objects 
\[
    \ob \catname{R}^A_k = \{\sigma \subseteq I \mid \# J_{\sigma} \geq k\}
\]
and 
\[
    \ob \catname{C}^A_l = \{\tau \subseteq J \mid \# I_{\tau} \geq l\}.
\]
While~$\nerve \catname{R}^A \simeq \nerve \catname{C}^A$ due to Dowker duality, the filtrations~$\nerve \catname{R}^A_{\bullet}$ and~$\nerve \catname{C}^A_{\bullet}$ can look quite different. A specific relation illustrates this difference. \\
Let~$\mathcal{U}=\{U_i \subseteq X\}_{i \in I}$ be an open cover of the topological space~$X$. We call~$\mathcal{U}$ an~$n$-\textit{fold} cover if every~$x \in X$ is contained in at least $n$ elements of~$\mathcal{U}$. We call~$\mathcal{U}$ an~$m$-\textit{good} cover if for all~$\sigma \subseteq I$ with~$\# \sigma \geq m$ the intersection~$U_{\sigma}$ is either empty or contractible. Now recall the relation~$A_{\mathcal{U}} \from I \times X \to \{0,1\}$ associated to the cover~$\mathcal{U}$ as in the introduction to this paper. 
\begin{proposition}\label{proposition:Rec1}
    Let~$1 \leq m \leq n$. If $\mathcal{U}$ is an~$n$-fold and~$m$-good open cover of~$X$ then for~$m \leq l \leq n$:
    \[
        \operatorname{B} \catname{C}^{A_{\mathcal{U}}}_l \simeq X
    \]
\end{proposition} 
\begin{proof}
    Define~$I_l = \{\sigma \subseteq I \mid \# \sigma \geq l\}$ and $\mathcal{U}_l = \{\cap_{i \in \sigma} U_i\}_{\sigma \in I_l}$. Because $\mathcal{U}$ is a~$n$-fold cover, we know that~$\mathcal{U}_l$ is an open cover of~$X$ for~$l \leq n$. Because~$\mathcal{U}$ is~$m$-good, for~$m \leq l$ the cover~$\mathcal{U}_l$ is good. Define a new relation 
\begin{align*}
    (A_{\mathcal{U}})_l : & I_l \times X \rightarrow \{0,1\} \\ 
    &(A_{\mathcal{U}})_l(\sigma,x) = \begin{cases} 1 \text{ if } x \in U_{\sigma} \\
    0 \text{ else}
    \end{cases}
\end{align*}
The row complex of this relation is precisely the \v{C}ech nerve~$\cechnerve \mathcal{U}_l$ and thus homotopy equivalent to~$X$ due to the nerve lemma. On the other hand $\operatorname{B} \catname{C}^{A_{\mathcal{U}}}_l$ is a subdivision of it's column complex. An application of (ordinary) Dowker duality finishes the proof. 
\end{proof}
On the other hand the filtration~$\cspace \catname{R}^{A_{\mathcal{U}}}_{\bullet}$ may never recover the homotopy type of~$X$. If we are however not given a relation between elements of an open cover and points of~$X$ but another collection of subsets~$\mathcal{V} = \{V_j \subseteq X\}_{j \in J}$ and a relation 
\begin{align*}
    A_{\mathcal{UV}} \from & I \times J \to \{0,1\} \\
    & A_{\mathcal{U V}}(i,j) = \begin{cases}
        1 \text{ if } U_i \cap V_j \neq \emptyset \\
        0 \text{ else}
    \end{cases}
\end{align*}
we can use total weights on the row complex to learn about the homotopy type of~$X$. 
\begin{proposition}\label{proposition:Rec2}
    Let~$\mathcal{U} = \{U_i \subseteq X\}_{i \in I}$ be a good open cover of~$X$. Let~$\mathcal{V} = \{V_j \subseteq X\}_{j \subseteq J}$ be a collection of subsets and~$A_{\mathcal{UV}} \from I \times J \to \{0,1\}$ a relation as above. If for every~$\sigma \subseteq I$ we have~$\#J^{A_{\mathcal{UV}}}_{\sigma} \leq p$ if $U_{\sigma} = \emptyset$ and~$\#J^{A_{\mathcal{UV}}}_{\sigma} \geq q$ if~$U_{\sigma} \neq \emptyset$ then for~$p \geq k \geq q$:
    \[
        \cspace \catname{R}^{A_{\mathcal{UV}}}_k \simeq X
    \]
\end{proposition}
\begin{proof}
   Define $J_k = \{\tau \subseteq J \mid \# \tau \geq k\}$ and a relation 
   \begin{align*}
    (A_{\mathcal{UV}})_{1,k} : & I \times J_k \rightarrow \{0,1\} \\ 
    &(A_{\mathcal{UV}})_{1,k}(i,\tau) = \begin{cases} 1 \text{ if } \forall j \in \tau : U_i\cap V_j \neq \emptyset\\
    0 \text{ else} 
    \end{cases}.
\end{align*} 
Then $\cspace \catname{R}^{A_{\mathcal{UV}}}_k$ is a subdivision of the row complex of this relation, which is precisely the \v{C}ech nerve of~$\mathcal{U}$.  
\end{proof}
The proofs of the above propositions suggests for a general relation~$A \from I \times J \to \{0,1\}$ the following definitions, which allow us to extract the filtrations~$\catname{R}^A_{\bullet}$ and~$\catname{C}^A_{\bullet}$ from appropriate filtrations of relations.
For~$k,l \geq 1$ we may define 
\[
    I_l = \{\sigma \subseteq I \mid \# \sigma \geq l\} \subseteq \mathcal{P}(I)
\]
and 
\[
    J_k = \{\tau \subseteq J \mid \# \tau \geq k \} \subseteq \mathcal{P}(J).
\]
From this we define 
\begin{align*}
    A_{k,l} & \from I_l \times J_k \rightarrow \{0,1\} \\
    & A_{k,l}(\sigma,\tau) = \begin{cases} 1 \quad \text{if} \quad \sigma \subseteq I_{\tau} \Leftrightarrow \tau \subseteq J_{\sigma} \\ 0 \quad \text{else} \end{cases}.
\end{align*}
There is then a functor of poset categories
\begin{align*}
    \psi_k \from & \catname{R}^{A_{k,1}} \rightarrow \catname{R}^A_k \\
    & \chi \mapsto \bigcup_{\sigma \in \chi} \sigma,
\end{align*}
which is well defined because~$\bigcap_{\sigma \in \chi} J_{\sigma} = J_{\bigcup_{\sigma \in \chi} \sigma}$.
\begin{proposition}
    For all~$k \geq 1$ the simplicial map~$\nerve \psi_k$ is a weak homotopy equivalence. 
\end{proposition}
\begin{proof}
    Let~$\sigma \in \ob \catname{R}^A_k$. Then the powerset~$\mathcal{P}(\sigma)$ gives us a terminal object for the fiber~$\psi/\sigma$ and Quillens Theorem A applies.  
\end{proof}
Similarly we obtain a weak equivalence~$\nerve \catname{C}^{A_{1,l}} \overset{\sim}{\rightarrow} \nerve \catname{C}^A_l$.
This suggests to study the bifiltration~$\nerve \catname{R}^{A_{\bullet,\bullet}} \simeq \nerve \catname{C}^{A_{\bullet,\bullet}}$. Combining Proposition \ref{proposition:Rec1} and Proposition \ref{proposition:Rec2} we obtain 
\begin{theorem}
    Let~$\mathcal{U} = \{U_i \subseteq X\}_{i \in I}$ be an~$n$-fold and~$m$-good open cover of the topological space~$X$. Let~$\mathcal{V} = \{V_j \subseteq X\}_{j \subseteq J}$ be a collection of subsets and~$A_{\mathcal{UV}} \from I \times J \to \{0,1\}$ a relation where~$A_{\mathcal{UV}}(i,j)=1$ if~$U_i \cap V_j \neq \emptyset$ and zero otherwise. If for every~$\sigma \subseteq I$ we have~$\#J^{A_{\mathcal{UV}}}_{\sigma} \leq p$ if $U_{\sigma} = \emptyset$ and~$\#J^{A_{\mathcal{UV}}}_{\sigma} \geq q$ if~$U_{\sigma} \neq \emptyset$ then for~$p \geq k \geq q$ and~$m \leq l \leq n$:
    \[
        \cspace \catname{R}^{(A_{\mathcal{UV}})_{k,l}} \simeq X
    \]
\end{theorem}
Even for moderate sized sets~$I$ and~$J$ the relations~$A_{k,l}$ however become extremely large and hard to work with in applications. The following is an attempt to organize the homotopical information in a more manageable format. 
\begin{theorem} \label{theorem:Application}
For every relation~$A \from I \times J \to \{0,1\}$ and natural numbers~$k \geq k'$ and~$l\geq l'$ there are zigzag of weak equivalences that make the diagram
\[
    \begin{tikzcd}
    \nerve \catname{R}^A_{k,l} \arrow[d] \arrow[r,dash,"\sim"] & \diag (\nerve \circ \catname{R}^{A_{1,l}A^{\top}_{k,l}}) \arrow[d] \\
    \nerve \catname{R}^A_{k',l'}  \arrow[r,dash,"\sim"] & \diag (\nerve \circ \catname{R}^{A_{1,l'}A^{\top}_{k',l'}}) 
    \end{tikzcd}
\]
    commute. 
\end{theorem}
\begin{proof}
    For~$\chi \in \ob \catname{R}^{A_{k,l}}$ we have~$\nerve \catname{C}^{\tilde{A}_{k,1}}_{(\tilde{J}_k)_{\chi}} \simeq \ast$ since~$(\tilde{J}_k)_{\chi}$ is a terminal object. Using Lemma \ref{lemma:SimpCatLevelwiseWeakEq} and Theorem \ref{theorem:FunctorialExtendedDowker} we then get natural weak equivalences
\[
    \nerve \catname{R}^{A_{k,l}} \simeq \diag \nerve \circ\catname{R}^{A_{k,l}A^{\top}_{k,1}} \simeq  \diag (\nerve \circ \catname{C}^{\tilde{A}_{k,1}^{\top}\tilde{A}_{k,l}}) = \diag (\nerve \circ \catname{R}^{\tilde{A}_{k,1}\tilde{A}^{\top}_{k,l}})  
\]
\end{proof}

\section{Outlook}
There is several directions in which to move forward. 
\paragraph{Implementation:} An obvious one is to design and implement an efficient algorithm to compute the persistence modules of the bifiltered Dowker complexes from Section \ref{section:BifiltDow}. Combined with modern software for the exploration of such multi persistence modules like \textit{Rivet} \cite{rivet} this could prove to be invaluable for the unsupervised exploration of large datasets for example in neuroscience. We hope that our Theorem \ref{theorem:Application} can be of use for such an implementation. 

\paragraph{Hypercovers and simplicial presheaves:} The proof of our main result (Theorem \ref{Theorem:main} is based on (a simplicial version of) Segal \cite{Segal}, Dugger and Isaksens \cite{Dugger} result about recovering a spaces homtopy type from any cover. In the latter paper a similar result about hypercovers is proven. It could be fruitful to apply this result in our context of Dokwer complexes. For example, given more than two composeable relations~$A,B,C,\ldots$ one could extend the constructions of section \ref{section:ExtendedDowker} to~$\catname{R}(A,B,C,\ldots)$ and~$\catname{C}(A,B,C,\ldots)$ and ask if a similar result~$\catname{R}(A,B,C,\ldots)\simeq \catname{C}(A,B,C,\ldots)$ holds. It could also be interesting to explore the connection with the homotopy theory of simplcial presheaves as in~\cite{dugger2004hypercovers}.

\paragraph{Hierarchical concept representation:}
Dowker complexes play an important role in the theory of formal concept analysis \cite{ayzenberg2019topology} \cite{freund2015lattice}. Our extension could be interesting in the analysis of hierarchical formal concepts. This provides another motivation for extending our theory further to accommodate more than two composeable relations. It also relates to the analysis of neural data with our methodology. An interesting research question is how hierarchical formal concepts are represented in the activity patterns of deep neural networks. Our constructions could provide a natural framework for studies exploring this.

\backmatter

\vspace{2cm}

\bmhead{Acknowledgments}
We want to thank \v{Z}iga Virk for helpful conversations that motivated our discussion of reconstruction results.  

\newpage

\bibliography{extended_dowker}

\end{document}